\documentclass[12pt]{amsart}
\usepackage{amsmath,amsthm,amssymb,latexsym,xy,amssymb,enumerate}
\xyoption{all}

\newtheorem{theorem}{Theorem}
\newtheorem{lemma}[theorem]{Lemma}
\newtheorem{corollary}[theorem]{Corollary}
\newtheorem{proposition}[theorem]{Proposition}
\newtheorem{conjecture}[theorem]{Conjecture}

\newtheorem{remark}[theorem]{Remark}

\newtheorem{question}[theorem]{Question}
\newcommand{\effdim}{\mathrm{eff.}\dim}

\title{Effective dimension of finite semigroups}

\author{Volodymyr Mazorchuk and Benjamin Steinberg}

\begin{document}

\begin{abstract}
In this paper we discuss various aspects of the problem of determining
the minimal dimension of an injective linear representation of a finite semigroup
over a field. We outline some general techniques and results, and apply
them to numerous examples.
\end{abstract}

\maketitle

\section{Introduction}
Most representation theoretic questions about a finite semigroup $S$ over a field $\Bbbk$ are really questions about the semigroup algebra $\Bbbk S$.  One question that is, however, strictly about $S$ itself is the minimum dimension of an effective linear representation of $S$ over $\Bbbk$, whereby effective we mean injective; we call this the effective dimension of $S$ over $\Bbbk$.  Note that semigroups (and in fact groups) with isomorphic semigroup algebras can have
different effective dimensions. For example, the effective dimension of $\mathbb{Z}/4\mathbb{Z}$ over $\mathbb{C}$
is $1$, whereas the effective dimension of  $\mathbb{Z}/2\mathbb{Z}\times \mathbb{Z}/2\mathbb{Z}$ is $2$
(since $\mathbb{C}$ has a unique element of multiplicative order two),
although both groups have algebras isomorphic to $\mathbb{C}^{4}$.

There are two natural questions that arise when considering the effective dimension of finite semigroups:
\begin{enumerate}[$($a$)$]
\item\label{question.a} Is the effective dimension of a finite semigroup decidable?
\item\label{question.b} Can one compute the effective dimension of one's favorite finite semigroups?
\end{enumerate}
These are two fundamentally different questions.  The first question asks for a Turing machine that on input the Cayley table of a finite semigroup, outputs the effective dimension over $\Bbbk$.  The second one asks for an actual number.  Usually for the second question one has in mind a family of finite semigroups given by some parameters, e.g., full (partial) transformation monoids, full linear monoids over finite fields, full monoids of binary relations, etc.  One wants to know the effective dimension as a function of the parameters.

The effective dimension of groups (sometimes called the minimal faithful degree) is a classical topic, dating back to the origins of representation theory.  There doesn't seem to be that much work in the literature on semigroups except for the paper~\cite{KiRo} of Kim and Roush and previous work~\cite{MS} of the authors.  This could be due in part to the fact that the question is much trickier for semigroups because semigroup algebras are rarely semisimple. Also, minimal dimension effective modules need not be submodules of the regular representation.

Question \eqref{question.a} has a positive answer if the first order theory of the field $\Bbbk$ is decidable.
Indeed, to determine the effective dimension of a finite semigroup one just needs to solve a finite collection of
systems of equations and inequations over $\Bbbk$ because the effective dimension is obviously bounded by the size of the semigroup plus one.  Classical results of Tarski imply that this is the case for algebraically closed fields
and for real closed fields.  However, the time complexity of these algorithms seems to be prohibitive to
applying them in practice.

Question \eqref{question.b} is answered for all the classical finite semigroups mentioned above, as well as several other families.

Our main tools are the classical representation theory of finite semigroups (as in~\cite[Chapter 5]{CP}, \cite{RZ} and \cite{GMS}), model theory, algebraic geometry and representation varieties, and George Bergman's lemma from~\cite{KiRo}.
We formulate various general techniques which can be used in the study of effective dimensions for certain classes
of semigroups and along the way recover and improve upon many partial results in this direction that were known, at least on the level of folklore, to representation and semigroup theorists.

The paper is organized as follows.   Section~\ref{s1} discusses elementary properties of effective modules and effective dimension.  In particular, the relevance of a classical result of R.~Steinberg~\cite{St} is discussed.  The main result in this section is an improvement on the obvious upper bound on effective dimension.  The next section explains how to interpret the effective dimension of a finite semigroup over a field $\Bbbk$ in the first order theory of $\Bbbk$ and then applies model theoretic results to deduce a number of conclusions, including decidability over algebraically closed and real closed fields.  Section~\ref{s2} gives a simple description of the effective dimension of a commutative inverse monoid over the complex field (or any sufficiently nice field) using Pontryagin duality for finite commutative inverse monoids.  These results in particular apply to finite abelian groups and to lattices, the former case of course being well known~\cite{Ka}.  The following section studies the effective dimension of generalized group mapping semigroups in the sense of Krohn and Rhodes, see~\cite{KR}.  This class includes full partial transformations monoids, symmetric inverse monoids, full binary relation monoids and full linear monoids over finite fields.  The effective dimension is computed in each of these cases.

Section~\ref{s4} discusses a lemma from~\cite{KiRo}, which is attributed to G.~Bergman.  Kim and Roush had already used this lemma (and a variant) to compute the effective dimension of the semigroups of Hall relations and reflexive relations.  The authors used it in previous work to compute the effective dimension of $0$-Hecke monoids associated to finite Coxeter groups, see~\cite{MS}.  In Section~\ref{s4}, we use it to compute the effective dimension of semigroups of transformations with a doubly transitive group of units and at least one singular transformation.   This applies in particular to full transformation monoids.  The next section studies the effective dimension of nilpotent semigroups.  Using elementary algebraic geometry and the notion of representation varieties, we show that generic $n$-dimensional representations of free nilpotent semigroups of nilpotency index $n$ are effective over an algebraically closed field.  It follows that the effective dimension of these semigroups is $n$.  The same is true for free commutative nilpotent semigroups of index $n$.  On the other hand, we construct arbitrarily large commutative nilpotent semigroups of any nilpotency index $n\geq 3$ with the property that effective dimension equals cardinality (this is the worse possible case).  This leads one to guess that the computational complexity of computing effective dimension for nilpotent semigroups should already be high.   Section~\ref{spath} computes the effective dimension, over an algebraically closed field, of various types of path semigroups, including the path semigroup of an acyclic quiver and certain truncated path semigroups.  Here, again, we use the technology of representation varieties.

The penultimate section considers some other examples that are essentially known in the literature, e.g., rectangular bands and symmetric groups, as well as some new results for hyperplane face semigroups and free left regular bands.
In the last section we present a table of effective dimensions over the complex numbers of various classical families of finite semigroups.

\medskip

\noindent
{\bf Acknowledgments.} An essential part of the paper
was written during a visit of the
second author to Uppsala University. The financial support and
hospitality of Uppsala University are gratefully acknowledged.
The first author was partially supported by the
Royal Swedish Academy of Sciences and the Swedish Research Council.
The second author was partially supported by NSERC.
We thank the referee for helpful comments.  At the time this research was performed, the second author was a member of the School of Mathematics and Statistics at Carleton University.

\section{Effective modules}\label{s1}

Let $\Bbbk$ be a field, $S$ a semigroup and $V$ a vector space over $\Bbbk$. A linear representation $\varphi\colon S\to \mathrm{End}_{\Bbbk}(V)$ is said to be \emph{effective} if it is injective.  We shall also say that the module $V$ is effective.  If, furthermore, the linear extension $\overline{\varphi}\colon \Bbbk S\to \mathrm{End}_{\Bbbk}(V)$ is injective, we say that $V$ is a \emph{faithful} module. Our choice of terminology follows \cite{GM} and is aimed at avoiding confusion between these two different notions.  Of course faithful modules are effective, but the converse is false.  For example, for $n>1$, the group $\mathbb{Z}/n\mathbb{Z}$ has an effective one-dimensional representation over $\mathbb{C}$ but no faithful one.

\subsection{Steinberg's theorem}
There is a well known result of R.~Steinberg (see \cite{St}) that says, in effect, that effective modules are not too far from faithful ones.  The result was generalized by Rieffel to the context of bialgebras~\cite{Ri}. Recall that,
given two $S$-modules $V$ and $W$, the vector space $V\otimes W$ has the natural structure of an $S$-module
given by
\begin{equation}\label{eq1}
s(v\otimes w):=sv\otimes sw.
\end{equation}
Technically speaking, $\Bbbk S$ is a bialgebra where the comultiplication $\bigtriangleup\colon \Bbbk S\to \Bbbk S\otimes \Bbbk S$ is given by $\bigtriangleup(s)=s\otimes s$ and the counit $\varepsilon\colon \Bbbk S\to \Bbbk$ given by $\varepsilon(s)=1$ for all $s\in S$. For bialgebras one can always define the tensor product of representations.

\begin{theorem}[R.~Steinberg]\label{Steinberg}
Let $V$ be an effective module for a semigroup $S$. Then
\begin{equation*}
\mathcal T(V)=\bigoplus_{n=0}^{\infty}V^{\otimes n}
\end{equation*}
is a faithful module.
\end{theorem}

One can draw the following consequences when $S$ is finite (an assumption we shall make for the remainder of the paper).

\begin{corollary}\label{Steinbergcor}
Let $S$ be a finite semigroup and $V$ be an effective module.
\begin{enumerate}[$($i$)$]
\item\label{Steinbergcor.1} There exists $k\ge 0$ such that
\begin{equation*}
\mathcal T^{k}(V)=\bigoplus^{k}_{n=0}V^{\otimes n}
\end{equation*}
is a faithful module.
\item\label{Steinbergcor.2}
Every simple $\Bbbk S$-module is a composition factor of
$V^{\otimes n}$ for some $n\geq 0$.
\end{enumerate}
\end{corollary}

\begin{proof}
To prove \eqref{Steinbergcor.1}, just observe that if $I_{j}\lhd \Bbbk S$, for $j=0,1,2,\dots$, is the annihilator ideal of $\mathcal{T}^{j}(V)$, then $I_{0}\supseteq I_{1}\supseteq\cdots$ and $\bigcap^{\infty}_{j=0} I_{j}=0$ by Theorem~\ref{Steinberg}.  Since $\Bbbk S$ is finite dimensional, it follows that $I_{k}=0$ for some $k\ge 0$.

Let $L$ be a simple $\Bbbk S$-module and suppose that $e$ is a primitive idempotent corresponding to the projective cover of $L$.  Let $k$ be as in \eqref{Steinbergcor.1}.  Then
\begin{equation*}
0\ne e\mathcal{T}^{k}(V)=\bigoplus^{k}_{n=0}eV^{\otimes n}.
\end{equation*}
Thus $eV^{\otimes n}\ne 0$ for some $n\ge 0$ and so $L$ is a composition factor of this tensor power.  This proves \eqref{Steinbergcor.2}.
\end{proof}

Note that if $V$ is not effective, then there exist distinct $s,t\in S$ such that $sv=tv$ for all
$v\in V$. Hence in this case formula \eqref{eq1} implies
\begin{displaymath}
s(v_1\otimes v_2\otimes\dots \otimes v_n)=t(v_1\otimes v_2\otimes\dots \otimes v_n)
\end{displaymath}
for all $v_1,v_2,\dots,v_n\in V$ and thus the module $\mathcal{T}(V)$ is neither faithful, nor effective.
Therefore, the conclusion of Theorem~\ref{Steinberg}, in fact, characterizes effectiveness of a module $V$.

\subsection{Elementary properties}

Let us define the \emph{effective dimension} $\effdim_{\Bbbk}(S)$ of a finite semigroup $S$ over $\Bbbk$ to be the minimum dimension of an effective module $V$.  By a \emph{minimal effective} module, we mean an effective module of dimension precisely $\effdim_{\Bbbk}(S)$.

If $S$ is a semigroup, then $S^{1}$ will denote the result of adjoining an external identity to $S$.
It is convenient to define
\begin{equation*}
S^{\bullet}=\begin{cases} S^{1}, & \text{if $S$ is not a monoid;}\\ S, & \text{else.}\end{cases}
\end{equation*}

If $A$ is a unital $\Bbbk$-algebra, then we say that an $A$-module $V$ is \emph{unital} if the identity of $A$ acts as the identity on $V$, i.e., the associated linear representation $\varphi\colon A\to \mathrm{End}_k(V)$ is a homomorphism of unital rings.

We record some elementary facts about effective dimension and minimal effective modules in the next proposition.

\begin{proposition}\label{elemfacts}
Let $S$ be a finite semigroup and $\Bbbk$ a field.
\begin{enumerate}[$($i$)$]
\item\label{elemfacts.1} If $\Bbbk S$ is unital (e.g., if $S$ is a monoid), then each minimal effective module is unital.
\item\label{elemfacts.2} If $S$ contains a zero element $z$, then $z$ annihilates each minimal effective module.
\item\label{elemfacts.3} $\effdim_{\Bbbk}(S)=\effdim_{\Bbbk}(S^{\bullet})\le |S^{\bullet}|$.
\item\label{elemfacts.4} If $T\le S$ is a subsemigroup, then $\effdim_{\Bbbk}(T)\le \effdim_{\Bbbk}(S)$.
\item\label{elemfacts.5} If $\mathbb{L}$ is a subfield of $\Bbbk$, then $\effdim_{\Bbbk}(S)\le \effdim_{\mathbb{L}}(S)$.
\item\label{elemfacts.6} $\effdim_{\Bbbk}(S)=\effdim_{\Bbbk}(S^{\mathrm{op}})$.
\end{enumerate}
\end{proposition}
\begin{proof}
If $e$ is an identity element for $\Bbbk S$ and $V$ is an effective module, then so is $eV$.  This proves \eqref{elemfacts.1}.  If $z$ is a zero element and $V$ is an effective module, then $V/zV$ is an effective module yielding \eqref{elemfacts.2}.
Claims \eqref{elemfacts.4}, \eqref{elemfacts.5} and \eqref{elemfacts.6} are trivial.
The inequality in \eqref{elemfacts.3} is established
by linearizing the left regular representation. If $S=S^{\bullet}$, the equality in \eqref{elemfacts.3} is obvious.
If $S\neq S^{\bullet}$, then $\effdim_{\Bbbk}(S)\leq \effdim_{\Bbbk}(S^{\bullet})$ by \eqref{elemfacts.4}.
Conversely, if $\varphi\colon S\to \mathrm{Mat}_{n\times n}(\Bbbk)$ is an effective representation, then the identity matrix is not in the image of $\varphi$ (for otherwise $S$ would be a monoid). Hence, if we extend $\varphi$ by mapping the identity of $S^{\bullet}$ to the identity matrix, then we obtain an effective representation of $S^{\bullet}$. This implies
$\effdim_{\Bbbk}(S^{\bullet})\leq \effdim_{\Bbbk}(S)$, thereby completing the proof.
\end{proof}

In light of Proposition~\ref{elemfacts}\eqref{elemfacts.1}, we shall always assume that all modules over a unital $\Bbbk$-algebra are unital.  Observe that Proposition~\ref{elemfacts}\eqref{elemfacts.3} allows us to reduce our study to monoids, which we shall do for  much of the remainder of the paper.

\subsection{A general upper bound}\label{s6}

Every finite semigroup $S$ has an obvious effective representation, namely, the
linearization of the left regular representation on $S^{\bullet}$, i.e., the  module $\Bbbk S^{\bullet}$.  This representation has
dimension $|S|$ if $S$ is a monoid and dimension $|S|+1$, if not.

The aim of this subsection is to prove the
following general result:

\begin{theorem}\label{geneffdim}
Let $S$ be a finite semigroup and $\Bbbk$ a field. Assume that the minimal ideal $I$ of $S$ is not a group, or that if it is a group, then the characteristic of $\Bbbk$ does not divide $|I|$.  Then $\effdim_{\Bbbk}(S)\le |S^{\bullet}|-1$, that is:
\begin{displaymath}
\effdim_{\Bbbk}(S)\leq
\begin{cases}
|S|-1, & \text{if $S$ is a monoid};\\
|S|, & \text{else}.
\end{cases}
\end{displaymath}
\end{theorem}
\begin{proof}
Let $I$ denote the minimal ideal of $S$ and let $e\in I$ be an idempotent.
Denote by $\mathcal{R}_e$, $\mathcal{L}_e$ and  $\mathcal{H}_e$ the corresponding
Green's classes containing $e$.

If $\mathcal{H}_e=I$ and the characteristic of $\Bbbk$ does not divide $|\mathcal{H}_{e}|$, then
\begin{equation*}
\eta:=\frac{1}{|\mathcal{H}_e|}\sum_{h\in \mathcal{H}_e}h
\end{equation*}
is a primitive idempotent of $\Bbbk S^{\bullet}$ and $\Bbbk S^{\bullet}\eta$ is isomorphic to the trivial $S$-module. This implies
that the trivial $S$-module is a direct summand of the effective representation $\Bbbk S^{\bullet}$.
Hence $\Bbbk S^{\bullet}/\Bbbk S^{\bullet}\eta$ is also effective and has the correct dimension.

If $I\ne \mathcal H_{e}$, then replacing $S$ by $S^{op}$ and applying
Proposition~\ref{elemfacts}\eqref{elemfacts.6} if necessary, we may assume $I\ne \mathcal{L}_{e}$.
This means that $I$ contains at least two different $\mathcal{L}$-classes.

Clearly the representation $\varphi$ associated to $\Bbbk S^{\bullet}/\Bbbk Se$
 separates elements in $S\setminus \mathcal{L}_e$
and also separates all such elements from elements of $\mathcal{L}_e$ (consider the action on the coset of the multiplicative identity).  Let $x\in I\setminus \mathcal{L}_{e}$.  If
$s,t\in \mathcal{L}_e$ are distinct, then $sx\ne tx$ since $I$ is a completely simple semigroup.  As $sx,tx\notin \mathcal{L}_{e}$, it follows that $\varphi(sx)\ne \varphi(tx)$, whence $\varphi(s)\ne \varphi(t)$. Thus $\varphi$ is effective of dimension strictly less than $|S^{\bullet}|$.  This completes the proof.
\end{proof}

The upper bound given by Theorem~\ref{geneffdim} is sharp.  For instance, none of the two-element semigroups without identity have an effective one-dimensional module.  On the other hand, a cyclic group of order $2$ has no effective one-dimensional module over a field of characteristic $2$.

\section{Decidability and other applications of model theory}\label{s8}

In this section we use \cite{Ma} as a general reference for model theory.
If $R$ is any ring, then the ring $\mathrm{Mat}_{n\times n}(R)$ of $n\times n$ matrices over $R$ is first order interpretable in $R$ (in the language of rings).  Elements of $\mathrm{Mat}_{n\times n}(R)$ can be identified with $n^{2}$-tuples of elements of $R$ by considering  matrix entries.  Equality of matrices is defined in terms of equality of their entries.  Matrix addition and multiplication can be expressed entrywise in terms of the ring operations.  Thus any first order statement about $\mathrm{Mat}_{n\times n}(R)$ is a first order statement about $R$.

Now let $S$ be a finite semigroup. The statement that $S$ has an effective representation of dimension $n$ over the ring $R$ is clearly a first order statement about $\mathrm{Mat}_{n\times n}(R)$.  One is asking for the existence of $|S|$ distinct elements of $\mathrm{Mat}_{n\times n}(R)$ that multiply according to the Cayley table of $S$.  It follows that if the first order theory of $R$ is decidable, then one can decide if $S$ has an effective representation of dimension $n$ (and the algorithm is uniform in $n$ and $S$, where $S$ is given by its Cayley table).

In light of the above remarks, classical model theory implies the following
collection of statements:

\begin{theorem}\label{modelthm}
Let $S$ be a finite semigroup.
\begin{enumerate}[$($a$)$]
\item\label{modelthm.1} If $\Bbbk$ and $\Bbbk'$ are two algebraically closed fields of the same
characteristic, then $\effdim_{\Bbbk}(S)=\effdim_{\Bbbk'}(S)$.
\item\label{modelthm.2} If $\effdim_{\mathbb{C}}(S)=n$, then $S$ has an effective
representation of dimension $n$ over some finite field.
\item\label{modelthm.3} If $S$ has (Krohn-Rhodes) complexity $n\ge 0$, then
$\effdim_{\mathbb{C}}(S)\ge n$.
\item\label{modelthm.4} $\effdim_{\Bbbk}(S)$ is effectively computable for any algebraically closed field $\Bbbk$.
\item\label{modelthm.5} $\effdim_{\Bbbk}(S)$ is effectively computable for any real closed field $\Bbbk$, e.g., the real numbers.
\end{enumerate}
\end{theorem}

\begin{proof}
Claim \eqref{modelthm.1} follows from the fact that the theory of an algebraically closed field of a given characteristic is complete and hence all models  are elementary equivalent, see the discussion following~\cite[Theorem~2.2.6]{Ma}.
Claim \eqref{modelthm.2} follows from the fact that a first order statement is true for
algebraically closed fields of characteristic zero if and only if  it is true for all
algebraically closed fields of sufficiently  large characteristic  (see \cite[Corollary~2.2.10]{Ma}) and the
observation that in positive characteristic a finite semigroup of matrices over the algebraic closure of the prime field is already defined over a finite field.
Claim \eqref{modelthm.3} follows from \eqref{modelthm.2} and the result that
the complexity of $\mathrm{Mat}_{n\times n}(\mathbb{F}_{q})$, for $n\ge 0$, is $n$ unless $q=2$, in which case the complexity is $n-1$,  see \cite[Theorem~4.12.31]{RS}.
Claim \eqref{modelthm.4} follows from the fact the first order theory of any algebraically closed field is decidable,
see \cite[Corollary~2.2.9]{Ma}. Hence the existence of an effective representation of $S$ of dimension
$n$ is decidable and we have only finitely many dimensions to check, namely $0,1,2,\dots,|S|$ (as the regular
representation of dimension $|S|+1$ is for sure effective).  Finally, claim \eqref{modelthm.5} follows as in the algebraically closed case because the first order theory of a real closed field is decidable, see~\cite[Corollary~3.3.16]{Ma}.
\end{proof}

From model theory it follows that the effective dimension of a finite semigroup $S$ over an algebraically closed field of characteristic zero agrees with its effective dimension over all algebraically closed fields of sufficiently large characteristic.  It is natural then to ask the following question.

\begin{question}
Is it true that if $S$ is a finite semigroup and $\Bbbk,\Bbbk'$ are algebraically closed fields whose characteristics do not divide the order of any maximal subgroup of $S$, then $\effdim_{\Bbbk}(S)=\effdim_{\Bbbk'}(S)$?
\end{question}

Another issue is the complexity of computing the effective dimension.  The algorithm for deciding the theory of algebraically closed fields is via quantifier elimination.  The sentence stating that $S$ has an effective representation of dimension $n$ belongs to the existential theory.  Our understanding from perusing the computer science literature is that for an algebraically closed field of characteristic zero, the existential theory is known to be NP-hard and to be in PSPACE. Since almost all semigroups are $3$-nilpotent (see \cite{KRS}), we conjecture the following:

\begin{conjecture}\label{conjn}
{\rm Computing the effective dimension over $\mathbb{C}$ of a $3$-nil\-po\-tent semigroup is NP-hard.}
\end{conjecture}

The first order theory of the field of rational numbers is undecidable.  This is a consequence of the solution of Hilbert's $10^{th}$ problem and a result of Julia Robinson that the integers are first order definable in the field of rational numbers.  To the best of our knowledge the existential theory of $\mathbb{Q}$ is open.  We do not know if effective dimension is decidable over $\mathbb{Q}$.  Of course effective dimension is computable over any given finite field.

In any event, it is not feasible in practice to compute the effective dimension of your favorite semigroup using the decidability of the first order theory.  Also an algorithm doesn't help for computing the effective dimension of an infinite family of semigroups, such as full transformation monoids or matrix monoids.  The rest of this paper focuses on more practical techniques for specific examples.

\section{Effective dimension of commutative inverse monoids}\label{s2}

In this section, we compute the effective dimension of a finite commutative inverse monoid  $M$ over sufficiently nice fields.   We shall see later that the situation is much more complicated for commutative semigroups in general.

Let $M$ be a finite monoid.  We shall say that a field $\Bbbk$ is a \emph{good splitting field} for $M$ if the following two conditions hold:
\begin{enumerate}[$($a$)$]
 \item The characteristic of $\Bbbk$ does not divide the order of any maximal subgroup of $M$.
 \item The field $\Bbbk$ is a splitting field for every maximal subgroup of $M$.
\end{enumerate}

If $\Bbbk$ is a good splitting field for a commutative inverse monoid $M$, then  $\Bbbk M\cong \Bbbk^{M}$ and hence $\Bbbk M$ is a semisimple algebra with one-dimensional simple modules, see~\cite[Chapter~5]{CP} or~\cite{Ste} for details.  This leads us to the consideration  of dual monoids in this context.

Fix a field $\Bbbk$.
Let $M$ be a finite monoid  and set
\begin{displaymath}
 \widehat{M}=\hom(M,\Bbbk)
\end{displaymath}
where we view $\Bbbk$ as a multiplicative monoid.  Then $\widehat{M}$ is a commutative inverse monoid under pointwise product called the \emph{dual monoid}  of $M$. It is not hard to verify that the natural homomorphism $\eta\colon M\to \widehat{\widehat{M}}$ is the universal map from $M$ to a commutative inverse monoid in the case that $\Bbbk$ is a good splitting field for $M$, and hence is an isomorphism if and only if $M$ is a commutative inverse monoid.  This Pontryagin duality for finite inverse monoids follows immediately from the classical representation theory of finite monoids; see~\cite[Chapter~5]{CP} or~\cite{RZ}.

The dual monoid of a finite commutative inverse monoid $M$ is easy to describe using the aforementioned representation theoretic results.  A classical result of Clifford says that commutative inverse semigroups are semilattices of abelian groups, see \cite[II.2]{Pe}.  More precisely, the $\mathcal H$-relation on a commutative inverse semigroup is a congruence and the quotient by this congruence is isomorphic to the semilattice of idempotents.  In the case of our finite $M$, the idempotents set $E(M)$ is a lattice.  It is straightforward to check that the $\mathcal H$-classes of $\widehat{M}$ are the dual groups of the $\mathcal H$-classes of $M$ and that the lattice $E(\widehat{M})$ is the opposite lattice of $E(M)$.

\begin{proposition}\label{seppoints}
Let $M$ be a finite commutative inverse monoid and $X\subseteq \widehat{M}$. Suppose that $\Bbbk$ is a good splitting field for $M$.  Then $X$ separates points of $M$ if and only if $X$ generates $\widehat{M}$ as a monoid.
\end{proposition}
\begin{proof}
Since $\widehat{M}$ separates points of $M$, so does any generating set.  Conversely, suppose that $X\subseteq \widehat{M}$ separates points.  Then the direct sum $V$ of the representations in $X$ is an effective module.  Corollary~\ref{Steinbergcor}\eqref{Steinbergcor.2} implies that each element of $\widehat{M}$ is a composition factor of some tensor power of $V$.  But the composition factors of $V^{\otimes n}$ are the elements of the $n$-fold product $X^{n}$.  Thus $X$ generates $\widehat{M}$.
\end{proof}

A representation of a commutative inverse monoid $M$ over a good splitting field is effective if and only if it is a direct sum of one-dimensional representations separating points.  Thus we have the following immediate consequence of Proposition~\ref{seppoints}.

\begin{theorem}\label{comminv}
Let $M$ be a finite commutative inverse monoid and $\Bbbk$ a good splitting field for $M$.  Then $\effdim_{\Bbbk}(M)$ is the minimum number of elements needed to generate $\widehat{M}$ as a monoid.
\end{theorem}

Since a finite abelian group is isomorphic to its dual, we recover the following standard fact (see e.g.
\cite[Theorem~4]{Ka}):

\begin{corollary}
If $G$ is a finite abelian group and $\Bbbk$ is a splitting field whose characteristic does not divide $|G|$, then $\effdim_{\Bbbk}(G)$ is its minimal number of generators.
\end{corollary}

Let $L$ be a finite lattice, viewed as a monoid via its meet.  Then $\widehat{L}$ can be identified with $L$ equipped with its join operation (the homomorphisms $L\to \Bbbk$ are the characteristic functions of principal filters).  A non-zero element $e\in L$ is \emph{join irreducible} if $e=f\vee g$ implies $e=f$ or $e=g$.  For example, if
\begin{equation*}
e_{0}< e_{1}<\cdots < e_{n}
\end{equation*}
is a chain of idempotents of length $n$, then all elements except $e_{0}$ are join irreducible.  If $X$ is a finite set, then the join irreducible elements of the power set of $X$ are the singletons.
It is well known and easy to prove that the join irreducible elements of a finite lattice form the unique minimal generating set under the join operation.  Thus Theorem~\ref{comminv} admits the following corollary.

\begin{corollary}\label{lattice}\label{latticecase}
If $L$ is a finite lattice viewed as a monoid via the meet operation, then the effective dimension of $L$ over any field is the number of join irreducible elements of $L$
\end{corollary}

In particular, we recover the following known (and easy) lower bound on effective dimension, where the length of a chain is defined to be one less than the number of elements in the chain.

\begin{corollary}\label{idembound}
If a finite semigroup $S$ contains a chain of idempotents of length $n$,
then $\effdim_{\Bbbk}(S)\geq n$ for any field $\Bbbk$.
\end{corollary}

There is a useful reformulation of Corollary~\ref{idembound}.  An old result of Rhodes~\cite{Rh2} shows that the length of the longest chain of idempotents in a finite semigroup is the same as the length of the longest chain of regular $\mathcal{J}$-classes.

\begin{corollary}\label{Jbound}
Let $S$ be a finite semigroup containing a chain of regular $\mathcal{J}$-classes of length $n$.
Then $\effdim_{\Bbbk}(S)\geq n$ for any field $\Bbbk$.
\end{corollary}

Let $\mathcal{PT}_{n}$ be the monoid of all partial transformations on an $n$-element set and $\mathcal{IS}_{n}$ the
submonoid of partial injective maps. The natural representations of both these semigroups are of degree $n$.  They also both have a chain of idempotents of length $n$.  Thus we have the folklore result that both these monoids have effective dimension $n$ over any field.

Let $L_{n}$ be the lattice consisting of a top, a bottom and $n$ incomparable elements.  Then it has $n$ join irreducible elements and so $L_{n}$ has effective dimension $n$.  Notice that both $\mathcal{PT}_{n}$ and $\mathcal{IS}_{n}$ contain a copy of $L_{n}$ consisting of the identity, the zero element (the nowhere defined map)
and all  the partial identity  transformations with singleton domains.

\section{Effective dimension of generalized group mapping semigroups}\label{s3}

The important notion of a generalized group mapping semigroup was introduced by Krohn and Rhodes in~\cite{KR}.  A non-trivial finite semigroup $S$ is called a \emph{generalized group mapping (GGM)} semigroup if it contains a ($0$-)minimal ideal $I$ on which it acts effectively on both the left and right.  This ideal is necessarily unique and regular and is called the \emph{distinguished ideal}  of $S$.  See~\cite[Chapter 8]{Ar} or~\cite[Chapter 4]{RS} for details.

If the distinguished ideal contains a non-trivial maximal subgroup, then the semigroup is called \emph{group mapping}; otherwise, it is called  \emph{AGGM}.  The ``A'' stands for aperiodic (a common term for finite semigroups with trivial maximal subgroups).  An AGGM semigroup must contain a zero element and so the distinguished ideal is $0$-minimal.

For example, $\mathcal{PT}_{n}$ and $\mathcal{IS}_{n}$ are AGGM semigroups, as is the monoid $\mathcal{B}_{n}$ of all binary relations on an $n$-element set.  Indeed, for both $\mathcal {PT}_n$ and $\mathcal{IS}_{n}$, the non-zero elements of the unique $0$-minimal ideal are the rank $1$ elements.  The action on the left of the set of rank $1$ partial identity maps is a copy of the original action in both cases, hence effective. (We assume partial transformations act on the left here.)  On the other hand the composition $1_{\{a\}}f$ is the constant map from $f^{-1}(a)$ to $a$ and so the action on the right of the $0$-minimal ideal is also effective in both cases.  If we view $\mathcal{B}_{n}$ as the monoid of $n\times n$ boolean matrices, then the non-zero elements of the unique $0$-minimal ideal are the products $v^T\cdot w$ of a non-zero column vector and with a non-zero row vector~\cite{KiRo}.  From this it is clear that $\mathcal{B}_{n}$ acts effectively on the left and right of this ideal.  The monoid $\mathrm{Mat}_{n\times n}(\mathbb{F}_{q})$ is group mapping if $q>2$ and is AGGM for $q=2$ (the unique $0$-minimal ideal consists of the rank $0$ and rank $1$ maps). The full transformation monoid $\mathcal{T}_{n}$ on an $n$-element set is not generalized group mapping.

It was shown by Rhodes that a non-trivial finite semigroup has an effective irreducible representation if and only if it is generalized group mapping and in the group mapping case, the non-trivial maximal subgroup of the distinguished ideal has an effective irreducible representation~\cite{Rh}.

In this section, we compute the effective dimension of an arbitrary AGGM semigroup and of certain group mapping semigroups. A key ingredient in the former case is the fact that an AGGM semigroup is subdirectly indecomposable, that is, has a unique minimal non-trivial congruence.
The following is~\cite[Theorem 4.20]{RS} (and its proof).

\begin{theorem}\label{sdi}
An AGGM semigroup is subdirectly indecomposable.  The unique minimal non-trivial congruence is the one collapsing the distinguished ideal to zero.
\end{theorem}

Let $S$ be a finite AGGM semigroup with distinguished ideal $I$.  Then by Rees's Theorem there is a $(0,1)$-matrix $P$ (say, of dimensions $m\times n$), called the \emph{structure matrix} of $I$, with no zero rows or columns, such that $I$ can be described as the semigroup of all $n\times m$-matrix units, together with the zero matrix, with multiplication given by $A\odot B=APB$.  The matrix $P$ is unique up to left and right multiplication by permutation matrices.  The fact that $I$ is itself AGGM implies that $P$ has distinct rows and columns; see~\cite[Proposition~4.7.14]{RS}. Let $z$ be the zero of $S$.  It follows from the classical representation theory of finite semigroups (see \cite{RZ}) that there is a unique simple $\Bbbk S$-module $M$ such that $zM=0$ and $IM\ne 0$.  Moreover, $\dim M$ is the rank of $P$.

\begin{theorem}\label{aggmeff}
Let $S$ be a finite AGGM semigroup with distinguished ideal $I$ and zero element $z$.  Let $M$ be the unique simple $\Bbbk S$-module with $zM=0$ and $IM\ne 0$.  Then $M$ is effective and is a composition factor of every effective module.  Thus $\effdim_{\Bbbk}(S)$ is the rank of the structure matrix of $I$.
\end{theorem}
\begin{proof}
The module $M$ is effective by Theorem~\ref{sdi} since the associated representation does not collapse $I$ to $z$.
If $V$ is an effective module, then so is $V/zV$ and so without loss of generality we may assume that $z$ annihilates $V$.  Let
\begin{equation*}
0=V_{0}\subseteq V_{1}\subseteq \cdots \subseteq V_{n}=V
\end{equation*}
be a composition series.  Let $e\in I$ be an idempotent with $e\neq z$.  Then there is a smallest $m$ such that $eV_{m}\ne 0$.  Then $e(V_{m}/V_{m-1})\cong eV_{m}/eV_{m-1}\cong eV_{m}\ne 0$ and hence $V_{m}/V_{m-1}\cong M$.  This completes the proof.
\end{proof}

For instance the structure matrices for the distinguished ideals of $\mathcal{PT}_{n}$ and $\mathcal{IS}_{n}$ have rank $n$.  This shows that the natural representations of these monoids are the unique minimal effective ones.  The distinguished ideal of $\mathrm{Mat}_{n\times n}(\mathbb{F}_{2})$ has rank $2^{n}-1$ by a result of Kov\'acs~\cite{Kovacs} (see the discussion below where general $q$ is considered). The unique minimal effective representation is then the linearization of the action by partial maps on $\mathbb{F}_{2}^{n}\setminus \{0\}$.

The results of Kim and Roush in \cite{KiRo} imply that the rank of the structure matrix for the distinguished ideal of $\mathcal B_{n}$ is $2^{n}-1$.  Thus our results slightly improve on their result by showing that $\mathcal{B}_{n}$ has a unique minimal effective module of this degree.

Let us now consider the case of a group mapping semigroup $S$ with zero element $z$.  The distinguished ideal $I$ is then isomorphic to a Rees matrix semigroup $\mathcal{M}^{0}(G,n,m,P)$ where $G$ is a maximal subgroup of $I\setminus \{z\}$ and $P$ is an $m\times n$-matrix with entries in $G\cup \{0\}$.  Here $n$ is the number of $\mathcal{R}$-classes and $m$ is the number of $\mathcal{L}$-classes of the $\mathcal{J}$-class $I\setminus \{z\}$.  The matrix $P$ is unique up to left and right multiplication by monomial matrices over $G$.  See~\cite{CP} for more on Rees matrix semigroups (but our notation follows~\cite[Appendix A]{RS}).  Let $\Bbbk$ be a field.  Then it is known that  $P$ is invertible over $\Bbbk G$ if and only if the algebra $\Bbbk I/\Bbbk z$ is unital, see~\cite[Chapter 5]{CP}.  More generally, $P$ has a left (resp.\ right) inverse if and only if $\Bbbk I/\Bbbk z$ has a left (resp.\ right) identity, again see~\cite[Chapter 5]{CP}.

As an example, consider the monoid $\mathrm{Mat}_{n\times n}(\mathbb{F}_{q})$ of $n\times n$-matrices over the finite field $\mathbb{F}_{q}$ of $q$ elements.  Assume that $q>2$.  Then $\mathrm{Mat}_{n\times n}(\mathbb{F}_{q})$ is a group mapping monoid with distinguished ideal the matrices of rank at most $1$.  The non-trivial maximal subgroup is isomorphic to $\mathbb{F}_{q}^{\times}$.  A result of Kov\'acs in \cite{Kovacs} implies that the structure matrix is invertible whenever the characteristic of $\Bbbk$ does not divide $q$.  The case of the complex field was first proved by Okni\'nski and Putcha in~\cite{PO}.

We shall need the following result, which is~\cite[Lemma~4.7.8]{RS}.

\begin{lemma}\label{injcrit}
Let $\varphi\colon S\to T$ be a semigroup homomorphism with $S$ a group mapping semigroup.  Let $G$ be a non-trivial maximal subgroup of the distinguished ideal.  Then $\varphi$ is injective if and only if its restriction to $G$ is injective.
\end{lemma}

We now compute the effective dimension of a group mapping semigroup when the structure matrix of the minimal ideal is left or right invertible.

\begin{theorem}\label{groupmapping}
Let $S$ be a group mapping semigroup with zero $z$ and distinguished ideal $I$.  Let $G$ be the maximal subgroup at an idempotent $e\in I\setminus \{z\}$.  Suppose that $\Bbbk$ is a field such that $\Bbbk I/\Bbbk z$ has a left (resp.\ right) identity (that is, the structure matrix of $I$ is left (resp.\ right) invertible over $\Bbbk G$).  Then
\begin{equation}\label{groupmapping.1}
\effdim_{\Bbbk}(S)=m\cdot \effdim_{\Bbbk}(G)
\end{equation}
where $m$ is the number of $\mathcal{R}$- (resp.\ $\mathcal{L}$)-classes of $I\setminus \{z\}$.
\end{theorem}
\begin{proof}
Replacing $S$ by $S^{op}$ if necessary, we may assume that $\Bbbk I/\Bbbk z$ has a left identity.
First we construct an effective representation of dimension the right hand side of \eqref{groupmapping.1}.  Let $A=\Bbbk S/\Bbbk z$.  Then $eAe\cong \Bbbk G$.   Recall the well known fact (cf.~\cite{GMS}) that $Ae$ is a free right $\Bbbk G$-module of rank $m$.  Suppose that $V$ is a minimal effective module for $G$ and put $W=Ae\otimes_{\Bbbk G} V$.  Then because $eW\cong V$ and $V$ is effective for $G$, it follows that $W$ is an effective module for $S$ by Lemma~\ref{injcrit}.  But $\dim W=m\cdot \dim V$.  This proves that
$\effdim_{\Bbbk}(S)\leq m\cdot \effdim_{\Bbbk}(G)$.  Notice that this part of the proof does not use left invertibility of the structure matrix.

Conversely, suppose that $W$ is a minimal effective module for $S$.  In particular, $zW=0$ by Proposition~\ref{elemfacts}\eqref{elemfacts.2}.  Thus $W$ is an $A$-module.  Then since $W$ is effective, it follows that $eW$ is effective for $G$. Let $V=Ae\otimes_{\Bbbk G} eW$.  Then  by the usual adjunction, one has $\hom_{\Bbbk G}(eW,eW)\cong \hom_{A}(V,W)$.  The homomorphism $\psi$ corresponding to $1_{eW}$ is defined on elementary tensors by $ae\otimes ew\mapsto aew$.   Notice that if $v\in V$, then $ev=e\otimes e\psi(v)$ since this is true on elementary tensors.  Thus $e$ annihilates $\ker \psi$ and hence so does $AeA\cong \Bbbk I/\Bbbk z$.  But the left identity of $\Bbbk I/\Bbbk z$ acts as the identity on $Ae=\Bbbk Ie/\Bbbk z$ and hence on $V$.  It follows that $\ker \psi=0$ and so $V$ is isomorphic to a submodule of $W$.  Thus  $\dim W\ge \dim V=m\cdot \dim eW$, which is at least as large as the right hand side of \eqref{groupmapping.1}.  This completes the proof of the theorem.
\end{proof}

Applying this result to the monoid of matrices over a finite field yields:

\begin{corollary}
Let $q=p^{k}$ with $p$ prime and suppose that $\Bbbk$ is a field of characteristic different than $p$ such that the polynomial $x^{q-1}-1$ has $q-1$ distinct roots in $\Bbbk$.  Then
\begin{equation*}
\effdim_{\Bbbk}(\mathrm{Mat}_{n\times n}(\mathbb{F}_{q}))=\frac{q^{n}-1}{q-1}.
\end{equation*}
\end{corollary}

\begin{proof}
We already proved this above for $q=2$.  Assume now that $q>2$.
The hypotheses of Theorem~\ref{groupmapping} hold by Kov\'acs theorem.  The distinguished ideal has maximal subgroup $\mathbb{F}_{q}^{\times}$, which has effective dimension $1$ over $\Bbbk$.  On the other hand the $\mathcal{R}$-class of a matrix is determined by its column space.  So the number of $\mathcal{R}$-classes of rank $1$ matrices is the number of lines in $\mathbb{F}^{n}_{q}$.  This completes the proof.
\end{proof}

It is easy to see that a group mapping inverse semigroup that is not a group has a zero.  The structure matrix can be chosen to be the identity in this case and so Theorem~\ref{groupmapping} applies.  For instance, it is known that any transitive subsemigroup of $\mathcal{IS}_{n}$ is generalized group mapping (this is due to Schein~\cite{Sch} with different terminology); see the discussion after~\cite[Proposition~4.9]{Ste2} and use that an inverse semigroup acts effectively on the left of an ideal if and only if it acts effectively on the right.  Thus the results of this section apply to such semigroups to compute the effective dimension.  Let us give some examples.

If $G$ is a non-trivial group, then the  partial transformation wreath product $G\wr \mathcal{IS}_{n}$ (in the sense of Eilenberg's book~\cite{E}) is a transitive inverse semigroup of partial bijections of $G\times \{1,\ldots, n\}$.  It is then a group mapping semigroup with maximal subgroup $G$ in the distinguished ideal.  One has
\begin{equation*}
\effdim_{\Bbbk}(G\wr \mathcal{IS}_{n})=n\cdot \effdim_{\Bbbk}(G).
\end{equation*}
For instance, the signed symmetric inverse monoid $\mathbb{Z}/2\mathbb{Z}\wr \mathcal{IS}_{n}$ has effective dimension $n$ over any field that is not of characteristic $2$.

Another example is the inverse semigroup
$\mathcal{P}\mathrm{Aut}(\mathbb{F}_q^n)$
of all partial linear bijections
on $\mathbb{F}_q^n$, studied in \cite{Ku}. It is again a transitive inverse semigroup of partial bijections.  It  is
group mapping with distinguished ideal
the maps of rank one and corresponding maximal
subgroup $\mathbb{F}_q^\times$. Therefore, from
Theorem~\ref{groupmapping} (and Theorem~\ref{aggmeff} if $q=2$) it follows that the
effective dimension of $\mathcal{P}\mathrm{Aut}(\mathbb{F}_q^n)$ over any splitting field for $\mathbb{F}_q^\times$
whose characteristic does not divide $q-1$ equals
$(q^n-1)/(q-1)$.

If $G$ is a non-trivial group, then the partial transformation wreath product $G\wr \mathcal{PT}_{n}$ is group mapping and the structure matrix has the block form
\begin{equation*}
\left(\begin{array}{cc} I_{n} \\ \ast \end{array}\right)
\end{equation*}
and hence is left invertible.  Thus Theorem~\ref{groupmapping} implies that
\begin{equation*}
\effdim_{\Bbbk}(G\wr \mathcal{PT }_{n})=n\cdot \effdim_{\Bbbk}(G).
\end{equation*}

\begin{remark}
{\rm
It is proved in~\cite{Rh2} that a finite semigroup $S$ has an effective completely reducible representation over a field of characteristic zero if and only if it is a subdirect product of generalized group mapping semigroups.  In this, case it follows from the results of~\cite{AMSV} that if $V$ is an effective module for $S$, then the direct sum of the composition factors of $V$ is also effective.  Thus $S$ has a minimal effective module that is semisimple.
}
\end{remark}

\section{Bergman's Lemma}\label{s4}

The following lemma was used by Kim and Roush in~\cite{KiRo}, where they attribute it to George Bergman.  They used it (and a variation) to compute the effective dimension of the semigroups of Hall matrices and reflexive binary relations.  The authors exploited this lemma in~\cite{MS} to compute the effective dimension of the $0$-Hecke monoid associated to a finite Coxeter group.

\begin{lemma}[Bergman]\label{bergmanlemma}
Let $S$ be semigroup and $L$ a left ideal of $\Bbbk S$ with simple socle.
Suppose that the socle of $L$ contains a non-zero element of the form $s-t$ with $s,t\in S$.
Then any $\Bbbk S$-module $V$ whose associated linear representation separates $s$ and $t$ contains $L$ as a submodule.
\end{lemma}
\begin{proof}
As $s-t$ does not annihilate $V$, there is an element $v\in V$ such that $(s-t)v\neq 0$.
The module homomorphism $L\to \Bbbk Sv$ given by $a\mapsto av$ must be injective because it does not annihilate the simple socle of $L$ (which is generated by $s-t$).
\end{proof}

As a corollary we compute the effective dimension of a family of monoids including full transformation monoids.

\begin{corollary}\label{doublytrans}
Let $S\leq \mathcal{T}_{n}$ have a doubly transitive group of units $G$ and contain a singular transformation.  Then $\effdim_{\Bbbk}(S)=n$ for any field $\Bbbk$ whose characteristic does not divide $|G|$.  Moreover, the natural representation of $S$ is a submodule of all effective modules.
\end{corollary}
\begin{proof}
It is well known that any submonoid of $\mathcal{T}_{n}$ with a doubly transitive group of units contains all the constant maps provided that it contains a singular transformation.  See for example~\cite{AS}.
Let $C$ be the set of constant maps.  It is a left ideal of $S$ and the action of $S$ on $\Bbbk C$ can be identified with the natural representation.  The subspace spanned by all differences $x-y$ with $x,y\in C$ is a submodule and is simple by double transitivity of $G$ on $C$ and classical group representation theory, cf.~\cite{F}.  Moreover, it is the unique proper submodule of $\Bbbk C$ since the only other composition factor of $\Bbbk C$ is the trivial module, which is the top but not a submodule.  An application of Lemma~\ref{bergmanlemma} completes the proof.
\end{proof}

The above corollary, in particular, applies to $\mathcal{T}_{n}$  itself and so $\mathcal{T}_{n}$ has effective dimension $n$ over any field of characteristic greater than $n$.

\section{Effective dimension of nilpotent semigroups}\label{s5}

In this section we study the effective dimension of nilpotent semigroups.  Recall that a semigroup $N$ is called \emph{nilpotent}  if it has a zero element $z$ such that $N^{n}=\{z\}$ for some $n$.  The minimum such $n$ is
termed the \emph{nilpotency index}  of $N$. The following lemma should be considered folklore.

\begin{lemma}\label{strictupper}
Every zero-preserving representation of a finite nilpotent semigroup $S$ is equivalent to one by strictly upper triangular matrices.
\end{lemma}
\begin{proof}
Let $z$ be the zero of $S$.  Any representation of $S$ mapping $z$ to $0$ is naturally a $\Bbbk S^{1}/\Bbbk z$-module $V$.  As each simple $\Bbbk S^{1}/\Bbbk z$-module is annihilated by the codimension-one nilpotent ideal $\Bbbk S/\Bbbk z$, taking a basis for $V$ adapted to a composition series establishes the result.
\end{proof}

Write $ST_{n}(\Bbbk)$ for the semigroup of strictly upper triangular $n\times n$-matrices over $\Bbbk$.  It is a nilpotent semigroup of nilpotency index $n$.  From Lemma~\ref{strictupper}  one can deduce the following easy but
useful result:

\begin{corollary}\label{cornilp}
If a semigroup $S$ has an element satisfying $s^{n}=s^{n+1}$ but $s^{n-1}\ne s^{n}$, then
$\effdim_{\Bbbk}(S)\geq n$.
\end{corollary}

For example, Corollary~\ref{cornilp} gives a lower bound $n$ for the
effective dimension of the Kiselman semigroup $K_n$ studied in \cite{KM}.
That this lower bound is in fact the exact value of the effective
dimension is proved in \cite{KM} by a subtle combinatorial argument.

An alternative way to prove Corollary~\ref{cornilp} is to observe that the minimal polynomial of $s$ has degree
$n$ and note that the minimal polynomial of any $k\times k$-matrix has degree at most $k$.

\subsection{Free nilpotent semigroups}\label{s5.1}

We assume that $\Bbbk$ is algebraically closed for this subsection so that we may apply results from algebraic geometry.  The reader is referred to~\cite[Chapter~1]{Ha} for basic notions from algebraic geometry.  The semigroup $ST_{n} (\Bbbk)$ is a linear algebraic semigroup, which moreover is an irreducible affine variety over $\Bbbk$.  In fact, it is isomorphic as a variety to affine $\binom{n}{2}$-space.   See the books of Putcha~\cite{Pu} and Renner~\cite{Re} for the theory of algebraic semigroups.

Let $N_{m,n}$ be the free $m$-generated nilpotent semigroup of nilpotency index (at most) $n$.  It is the quotient of the free semigroup by the ideal of words of length at least $n$ and hence is finite.  By Lemma~\ref{strictupper}, a representation of $N_{m,n}$ amounts to an $m$-tuple of elements of $ST_{n}(\Bbbk)$.  The space of such $m$-tuples is again an irreducible affine variety (called the \emph{representation variety} of $N_{m,n}$), this time isomorphic to affine $m\binom{n}{2}$-space.    We will show that the set of $m$-tuples corresponding to effective representations of $N_{m,n}$ is a non-empty Zariski open subset.  Thus generic representations of dimension $n$ are effective.  Since $N_{m,n}$ cannot be effectively represented in smaller degree, this will yield that $n$ is the effective dimension of $N_{m,n}$ over any algebraically closed field.

\begin{theorem}\label{genericallyfree}
Let $\Bbbk$ be an algebraically closed field.
The subset of $ST_{n}(\Bbbk)^{m}$ corresponding to effective representations of $N_{m,n}$ is a non-empty Zariski open subset and hence Zariski dense.
\end{theorem}

\begin{proof}
If $(A_{1},A_{2},\ldots,A_{m})\in ST(\Bbbk)^{m}$ and $u\in N_{m,n}$, then the image of $u$ under the representation corresponding to this $m$-tuple will be denoted $u(A_{1},A_{2},\ldots,A_{m})$.  For $u\ne v\in N_{m,n}$, let $V_{u,v}$ be the algebraic set defined by the polynomial equations  $u(A_{1},A_{2},\ldots,A_{m})=v(A_{1},A_{2},\ldots,A_{m})$.  Then the algebraic set of non-effective representations of $N_{m,n}$ is
\begin{equation}\label{genericallyfree.1}
V=\bigcup_{u\ne v\in N_{m,n}} V_{u,v}.
\end{equation}
Note that the union in \eqref{genericallyfree.1} is finite as $N_{m,n}$ is finite.  Thus, since $ST_{n}(\Bbbk)^{m}$ is irreducible, to prove our result it suffices to show that each $V_{u,v}$ with $u\ne v$ is proper.

Without loss of generality, we may assume that $u$ is non-zero and that if $v$ is a word, then $|u|\le |v|$.  Let $X$ be the free generating set for $N_{m,n}$ and assume that $u=x_{1}\cdots x_{k}$ with the $x_{i}\in X$ and $k<n$. Define a representation $\varphi\colon N_{m,n}\to ST_{n}(\Bbbk)$ on $x\in X$ by
\begin{equation*}
\varphi(x)_{ij} = \begin{cases} 1, & \text{if}\ x=x_{i},\ j=i+1;\\ 0, & \text{else.}\end{cases}
\end{equation*}
Clearly, $\varphi(x)\in ST_{n}(\Bbbk)$.  Observe that the only elements of $N_{m,n}$ that are not mapped to zero under $\varphi$ are the factors of $u$.  Indeed, if $w$ is a word of length less than $n$, then for $i<j$
\begin{displaymath}
\varphi(w)_{ij}= \begin{cases} 1, & \text{if}\ w=x_i\cdots x_{j-1}\\ 0, & \text{else.}\end{cases}
\end{displaymath}
In particular, $v$ maps to $0$ and $u$ does not.  Thus $V_{u,v}$ is a proper subset.  This completes the proof.
\end{proof}

For example, the elements of $ST_{n}(\mathbb{C})^{m}$ with strictly positive integer entries above the diagonal form a Zariski dense subset.  It follows from Theorem~\ref{genericallyfree} that there are effective $n$-dimensional representations of $N_{m,n}$ by strictly upper triangular matrices with positive integer entries above the diagonal.

\begin{corollary}\label{freenilp}
The effective dimension of the free nilpotent semigroup $N_{m,n}$ of index $n$ over an algebraically closed field is $n$.
\end{corollary}

\subsection{Commutative nilpotent semigroups}\label{s5.2}
Our next theorem is the analogue of Theorem~\ref{genericallyfree} for free commutative nilpotent semigroups of index at most $n$.

\begin{theorem}\label{freecomm}  Let $CN_{m,n}$ be the free commutative nilpotent
semigroup on $m$ generators of nilpotency at most index $n$.  Then the
effective dimension of $CN_{m,n}$ is $n$ for any field $\Bbbk$ that is
either of characteristic $0$ or algebraically closed.
\end{theorem}

\begin{proof}
The lower bound is clear since $CN_{m,n}$ has nilpotent elements of
index $n$.  For the converse, let $J$ be the $n\times n$ nilpotent
Jordan block.  Assume first that $\Bbbk$ has characteristic $0$.  Let
$p_1,\ldots, p_m$ be distinct prime positive integers.
Since the primes generate a
free commutative semigroup, trivially the semigroup generated by
$p_1J,\ldots, p_mJ$ is isomorphic to $CN_{m,n}$.

Next assume that $\Bbbk$ is algebraically closed.   Let $x_1,\ldots, x_m$ be
the free generating set for $CN_{m,n}$. To each element $(a_1,\ldots,
a_m)$ of $\Bbbk^m$ we associate the representation of $CN_{m,n}$
sending the $x_i$ to $a_iJ$, in this way making $\Bbbk^m$ a
representation variety.  We claim that the set of $m$-tuples
corresponding to effective representations of $CN_{m,n}$ is a
non-empty Zariski open subset of this affine space. Indeed, if $u,v\in CN_{m,n}$ are distinct, then the set $V_{u,v}$  of
representations from this representation variety that do not separate $u,v$ is an algebraic set.  The set of non-effective representations is the union of the algebraic sets $V_{u,v}$ running over the finitely many pairs of distinct elements $u,v\in CN_{m,n}$.  Since $\Bbbk^{m}$ is irreducible, it suffices to show that each $V_{u,v}$ is proper.  If one of $u,v$ is $0$, then the
representation associated to $(1,\ldots,1)$ does the job.  If neither
is zero, then there is a generator $x_j$ that appears a different
number of times in $u$ than in $v$.  Choose an element $a\in
\Bbbk^{\times}$ of order at least $n$.  Then the representation
corresponding to the $m$-tuple with $a_j=a$ and all other $a_i=1$
separates $u$ and $v$.
\end{proof}

In light of Theorems~\ref{genericallyfree} and \ref{freecomm}
one might think that any nilpotent semigroup of index $n$ has an effective representation of dimension $n$.  If $n\le 2$, this is true since all nilpotent semigroups of these indices are free of their index.  For $n>2$, this is false, as we proceed to show.
Our main tool is the following technical lemma.

\begin{lemma}\label{partinj}
Suppose that $N$ is a finite nilpotent semigroup with zero element $z$ such that:
\begin{enumerate}[$($i$)$]
\item\label{partinj.1} $N$ acts by partial injective maps on the left of $N\setminus \{z\}$;
\item\label{partinj.2} there is a unique element $w\in N\setminus \{z\}$ with $Nw=\{z\}$.
\end{enumerate}
Then every effective module  contains $\Bbbk N^{1}(1-z)\cong \Bbbk N^{1}/\Bbbk z$ as a submodule, whence  $\effdim_{\Bbbk}(N)=|N|$.
\end{lemma}
\begin{proof}
We claim that $\Bbbk N^{1}(1-z)$ has simple socle spanned by $w-z$.   The result will then follow from  Lemma~\ref{bergmanlemma} applied to the monoid $N^{1}$. Note that \eqref{partinj.2} implies that $w-z$ generates the simple $\Bbbk N^{1}$-module that is annihilated by $N$.   Because $\Bbbk N^{1}=\Bbbk N^{1}(1-z)\oplus \Bbbk z$, it suffices to show that the only simple submodules of $\Bbbk N^{1}$ are $\Bbbk z$ and $\Bbbk(w-z)$.

So suppose that
\begin{equation*}
\alpha=\sum_{s\in N^{1}}c_{s}s,
\end{equation*}
with the $c_s\in\Bbbk$,
generates a simple submodule. If it generates the trivial module, then from $z\alpha=\alpha$, we see that $\alpha\in \Bbbk z$.  Otherwise, $N\alpha=0$.  Suppose that $s\notin \{w,z\}$.  Let $k$ be the nilpotency index of $N$.  Then since $N^{k}s=\{z\}$, there exists $u\in N$ such that $us=w$ by \eqref{partinj.2}.  As $N$ acts on $N^{1}\setminus \{z\}$ by partial injective maps, it follows that the coefficient of $w$ in $0=u\alpha$ is $c_{s}$.  Thus $c_{s}=0$ for $s\notin \{w,z\}$.  On the other hand, $0=z\alpha=c_{w}z+c_{z}z$.  Thus $\alpha\in \Bbbk(w-z)$, as required.  This completes the claim.
\end{proof}

Let $NC_{m}$ be the free semigroup with zero on $m$ generators satisfying the identities $xy=yx$ and $x^{2}=0$.  It is not difficult to see that if $X$ is an $m$-element set, then $NC_{m}$ can be identified with the non-empty subsets of $X$ together with a zero element $z$. In particular, $|NC_{m}|=2^{m}$.       If $A,B$ are non-empty subsets of $X$, then
\begin{equation*}
AB=\begin{cases} A\cup B, & \text{if}\ A\cap B=\emptyset;\\ z, & \text{else.}\end{cases}
\end{equation*}
In particular, $NC_{m}$ is nilpotent of index $m+1$.    From this description, it is clear that $NC_{m}$ satisfies the hypotheses of Lemma~\ref{partinj} with $w=X$.  Thus we have the following result.

\begin{proposition}
The semigroup $NC_{m}$ has nilpotency index $m+1$ and effective dimension $2^{m}$ (over any field).
\end{proposition}

This example can be bootstrapped as follows. Let $A_{1},\ldots, A_{k}$ be disjoint finite sets of cardinalities $m_{1},\ldots,m_{k}$.  Let $S$ consist of all proper non-empty subsets of the $A_{i}$, for $1\le i \le k$, together with elements $z,w$.  Define a binary operation on  $S$ by making $z$ a zero element, by putting $Sw=wS=\{z\}$ and setting, for $\emptyset\ne X\subsetneq A_{i}$ and $\emptyset\ne Y\subsetneq A_{j}$,
\begin{equation*}
XY=\begin{cases}X\cup Y, & \text{if}\ i=j,\ X\cap Y=\emptyset\ \text{and}\ X\cup Y\subsetneq A_{i}; \\ w, & \text{if}\ i=j,\ X\cap Y=\emptyset\ \text{and}\ X\cup Y= A_{i};\\ z, & \text{else.} \end{cases}
\end{equation*}
Then $S$ is a commutative nilpotent semigroup of nilpotency index $\max\{m_{i}:1\le i\le k\}+1$ and size
\begin{equation}\label{bigdim}
2-2k+\sum^{k}_{i=1}2^{m_{i}}
\end{equation}
satisfying the hypotheses of Lemma~\ref{partinj}.  Thus $S$ has effective dimension given by \eqref{bigdim}.  We therefore have the following result.

\begin{proposition}
For any positive integer $n>2$, there are arbitrarily large commutative nilpotent semigroups of nilpotency index $n$ with effective dimension equal to cardinality.
\end{proposition}

Recall from Theorem~\ref{geneffdim} that for a semigroup that is not a monoid, the cardinality is an upper bound on effective dimension.  Thus the above result is the best possible.

\section{Effective dimension of path semigroups}\label{spath}

\subsection{Acyclic path semigroups}\label{spath.1}

Let $Q$ be a  quiver (equals directed graph).  The \emph{path semigroup} $P(Q)$ of $Q$ is the semigroup consisting of all the (directed) paths in $Q$, including an empty path $\varepsilon_x$ at each vertex $x$, together with a zero element $z$.
The product of two paths is given by concatenation, when it makes sense, and by $z$, otherwise. The semigroup $P(Q)$ is finite if and only if $Q$ is finite
and acyclic, which we assume in this subsection. The \emph{path algebra} $\Bbbk Q$ of the quiver $Q$ is the algebra $\Bbbk P(Q)/\Bbbk z$.  This algebra is unital with identity the sum of the empty paths.  In this section, we compute the effective dimension of $P(Q)$, under the assumption that $\Bbbk$ is algebraically closed, via the study of representation varieties.

Let us write $Q_{0}$ and $Q_{1}$ for the vertex and edge sets of $Q$, respectively.   We shall write $s(e)$ for the source of an edge and $t(e)$ for the target.  We shall also use this notation for paths.  We recall some basic notions from the representation theory of quivers, see~\cite{B} for details.  A \emph{representation} $\rho$ of a quiver $Q$ is an assignment of a (finite dimensional) $\Bbbk$-vector space $V_{x}$ to each vertex $x\in Q_{0}$ and a linear transformation $\rho_{e}\colon V_{t(e)}\to V_{s(e)}$ to each $e\in Q_{1}$.  Let $n_{x}=\dim V_{x}$.  Then $\boldsymbol n_{\rho}:=(n_{x})_{x\in Q_{0}}$ is termed the \emph{dimension vector} of $\rho$.  If $p=e_{1}\cdots e_{m}$ is a path (perhaps empty), then let $\rho_{p}\colon V_{t(p)}\to V_{s(p)}$ be the composition $\rho_{e_{1}}\cdots \rho_{e_{m}}$.

 The categories of representations of the quiver $Q$ and of $\Bbbk Q$-modules are equivalent.  The $\Bbbk Q$-module associated to a representation $\rho$ of $Q$ has underlying vector space
\begin{equation*}
V_{\rho}:=\bigoplus_{x\in Q_{0}} V_{x}.
\end{equation*}
The action of a path $p$ is obtained by extending $\rho_{p}$ to be $0$ on the summands other than $V_{t(p)}$.  It follows that $V_{\rho}$ is effective if and only if, for any two coterminous paths $p\ne q$, one has $\rho_{p}\ne \rho_{q}$ and also $\rho_{p}\ne 0$ for all paths $p$. Note that
\begin{equation*}
\dim V_{\rho}=\sum_{x\in Q_{0}}n_{x}.
\end{equation*}
Conversely, if $V$ is a $\Bbbk Q$-module, then a representation $\rho$ of $Q$ is obtained by setting $V_x=\varepsilon_xV$ for $x\in Q_0$ and $\rho_e=\varphi(e)|_{V_{t(e)}}$ where $\varphi\colon P(Q)\to \mathrm{End}_k(V)$ is the associated representation.

If $\boldsymbol n=(n_{x})_{x\in Q_{0}}$ is a dimension vector, the corresponding \emph{representation variety} is
\begin{equation*}
\mathcal{V}(\boldsymbol n):=\prod_{e\in Q_{1}}\mathrm{Mat}_{n_{s(e)}\times n_{t(e)}}(\Bbbk)
\end{equation*}
which is an irreducible affine variety isomorphic to an affine space of dimension $\sum_{e\in Q_{1}} n_{s(e)}n_{t(e)}$.  It is clear that elements of $\mathcal{V}(\boldsymbol n)$ are in bijection with quiver representations sending $x\in Q_{0}$ to $\Bbbk^{n_{x}}$.

We are now ready to compute the effective dimension of the path semigroup of an acyclic quiver.

\begin{theorem}\label{thmpath}
Let $Q$ be a finite acyclic quiver with n vertices. Then $\effdim_{\Bbbk}(P(Q))=n$ for any algebraically closed field $\Bbbk$. More precisely, if $\boldsymbol 1$ denotes the dimension vector with $n_{x}=1$ for all $x\in Q_{0}$, then the effective representations in the representation variety $\mathcal{V}(\boldsymbol 1)$ form a non-empty Zariski open subset.
\end{theorem}

\begin{proof}
The subsemigroup $S$ of $P(Q)$ consisting of the empty paths and the zero $z$ is a meet semilattice in which the empty paths form an anti-chain of size $n$.  It follows that the join irreducible elements of the lattice $S^{1}$ are the empty paths and so
\begin{equation*}
\effdim_{\Bbbk}(P(Q))\ge \effdim_{\Bbbk}(S)=\effdim_{\Bbbk}(S^{1})=n
\end{equation*}
by the results of Section~\ref{s2}.  This yields the lower bound.

If $p\ne q$ are coterminous paths in $Q$, let $V_{p,q}$ be the algebraic set consisting of those $\rho\in \mathcal{V}(\boldsymbol 1)$ such that $\rho_{p}=\rho_{q}$ (this subset is clearly polynomially defined).  Let $V_{p}$ be the algebraic set of representations $\rho$ such that $\rho_{p}=0$ for a path $p$.  Then the set of non-effective representations with dimension vector $\boldsymbol 1$ is the Zariski closed subset which is the union of the $V_{p,q}$ running over the finite set of all distinct pairs $p,q$ of coterminous paths and the $V_{p}$ running over the finite set of paths $p$.  Thus, by irreducibility of $\mathcal{V}(\boldsymbol 1)$, to complete the proof of the theorem it suffices to show that each $V_{p,q}$ and $V_{p}$ is a proper subset.

If $p$ is a path, then the representation given by $\rho_{e}=(1)$ for all edges $e$ does not belong to $V_{p}$.  Next assume that $p,q$ are distinct coterminous paths.  Since $Q$ is acyclic, a path is uniquely determined by its set of edges.  We may assume without loss of generality that there is an edge $e$ that appears in $p$ and not in $q$.  Define a representation $\rho\in \mathcal{V}(\boldsymbol 1)$ by putting $\rho_{e}=(0)$ and $\rho_{f}=(1)$ for all edges $f\ne e$.  Then $\rho_{p}=(0)$ and $\rho_{q}=(1)$.  Therefore, $\rho\notin V_{p,q}$, as required.
\end{proof}

\subsection{Truncated path semigroups}\label{spath.2}

Let $Q$ be a finite quiver with
$n$ vertices.  Let $N\in\mathbb{N}$ and $J$ denote the ideal of $P(Q)$ generated by all paths
of length one (i.e., all arrows). The semigroup $P(Q)/J^N$ is called the
{\em truncated path semigroup}.  Let $z$ denote the zero of this semigroup.  The following result generalizes Theorem~\ref{genericallyfree}.

\begin{theorem}\label{cr51}
Let $Q$ be a finite quiver. Assume that every vertex of $Q$ appears as
a vertex in some oriented cycle (or loop). Then
$\effdim_{\Bbbk}(P(Q)/J^N)=Nn$ for any al\-geb\-ra\-i\-cal\-ly
closed field $\Bbbk$.
\end{theorem}
\begin{proof}
Set $S:=P(Q)/J^N$ and let $A$ denote the truncated path algebra
$\Bbbk Q/\Bbbk J^N$. Assume first that $Q$ is an oriented cycle (or loop).
Then $A$ is a self-injective algebra, moreover, for every vertex $x$
the socle of the indecomposable projective-injective module $A\varepsilon_x$
is generated by the unique longest path $p_{x}$ starting at $x$.  Identifying $A$ with the algebra $\Bbbk S(1-z)$ and using that $\Bbbk S=\Bbbk S(1-z)\oplus \Bbbk z$, we see that the simple socle of $A\epsilon_{x}$ is generated by $p_{x}-z$ in $\Bbbk S$.
Therefore Lemma~\ref{bergmanlemma} implies that every indecomposable
projective-injective $A$-module is a submodule (and hence a direct summand)
of every effective $S$-module. Thus the regular representation of
$A$ is a direct summand of every effective module, which yields
$\effdim_{\Bbbk}(S)=Nn$ and, further, the inequality
\begin{equation}\label{eq123}
\dim \varepsilon_x V\geq N
\end{equation}
for any effective $S$-module $V$.

Assume now that $Q$ is a finite quiver such that every vertex of $Q$
appears as a vertex in some oriented cycle (or loop). Let $V$ be an
effective $S$-module. For a vertex $x$
let $Q'$ be a shortest oriented cycle in which $x$ appears and let
$T$ be the subsemigroup of $S$ generated by all arrows and
empty paths of $Q'$. Then \eqref{eq123} applied to $T$ implies
$\dim \varepsilon_x V\geq N$ and hence
$\effdim_{\Bbbk}(S)\geq Nn$.

On the other hand, if we set $n_x=N$ for all $x\in Q_0$ and let $\boldsymbol{N}$ be the corresponding dimension vector, then
\begin{displaymath}
\mathcal W := \prod_{e\in Q_1}ST_N(\Bbbk)
\end{displaymath}
is a closed subvariety of $\mathcal V(\boldsymbol N)$.  Note that $\mathcal W$ is irreducible because it is isomorphic to an affine space of dimension $|Q_1|\cdot \binom{N}{2}$.  Clearly the $\Bbbk Q$-module associated to any representation in $\mathcal W$ is a $\Bbbk Q/\Bbbk J^N$-module of dimension $nN$.  The proof of Theorem~\ref{genericallyfree} then applies \textit{mutatis mutandis} to show that the effective representations of $S$ form a non-empty Zariski open subset of $\mathcal W$.  This proves the claim.
\end{proof}

The hypothesis that each vertex belongs to an oriented cycle (or loop) is equivalent to the assumption that each strongly connected component of $Q$ contains at least one edge.
In the previous theorem, the proof of the upper bound does not require that each vertex of $Q$ appears in an oriented cycle, but the proof of the lower bound does.  This leads to the following question.

\begin{question}\label{qtruncated}
Let $Q$ be a finite quiver.  What is $\effdim_{\Bbbk}(P(Q)/J^N)$ for an algebraically closed field $\Bbbk$?
\end{question}

Quite a bit is known about the geometry of representation varieties for truncated path algebras, see~\cite{BHT}, and perhaps this can be of use for solving Question~\ref{qtruncated}.

\subsection{Incidence algebras}\label{spath.3}

Another class of path semigroups which is easily covered by our methods is the class of incidence algebras. Let $Q$ be a finite acyclic quiver and
$\sim$ the congruence on $P(Q)$ defined as follows: $p\sim q$ if and only
if $p$ and $q$ are coterminous. Vertices of the quiver $Q$ form a poset defined as follows: $x\leq y$ if and only if there is a path in $P(Q)$ from $x$ to $y$.
The quotient $I(Q):=P(Q)/{\sim}$ is called the {\em incidence semigroup} of this poset.  Notice that all finite posets $P$ arise in this way since we can take $Q$ to be the oriented Hasse diagram of $P$. Also, $I(Q)$ depends only on the poset structure and not $Q$.  The algebra $\Bbbk I(Q)/\Bbbk z$ is usually called the \emph{incidence algebra} of $Q$, or more precisely of the associated poset.

\begin{theorem}\label{cor52}
Let $Q$ be a finite acyclic quiver.  Then $\effdim_{\Bbbk}(I(Q))=n$ for any field $\Bbbk$.
\end{theorem}
\begin{proof}
That $\effdim_{\Bbbk}(I(Q))\geq n$ is proved similarly to
the first part of the proof of Theorem~\ref{thmpath}. On the other hand,
the {\em natural} $n$-dimensional representation of $Q$ that assigns
$\Bbbk$ to every vertex and the identity linear transformation to
every arrow obviously induces an effective module for $I(Q)$. The claim follows.
\end{proof}

\section{Bands and other examples}\label{s7}
In this section we consider a miscellany of other examples, starting with some bands.

\subsection{Rectangular bands}\label{s7.1}
Recall that a \emph{band} is a semigroup in which each element is idempotent.
Let us denote by $R_{m,n}$, for $m,n\in\mathbb{N}$, the {\em $m\times n$-rectangular band}, which is the semigroup
with underlying set $\{1,\ldots,m\}\times \{1,\ldots, n\}$ and multiplication $(i,j)\cdot(i',j')=(i,j')$.  The representation theory of $2\times 2$-rectangular bands is essentially worked out in~\cite{BMM}.

\begin{proposition}\label{rectband}
Let $m,n\in\mathbb{N}$.
\begin{enumerate}[$($a$)$]
\item\label{rectband.1} If $m=n=1$, then $\effdim_{\Bbbk}(R_{m,n})=0$ for any field $\Bbbk$.
\item\label{rectband.2} If $m=1$ and $n\neq 1$, or $n=1$ and $m\neq 1$, then
$\effdim_{\Bbbk}(R_{m,n})\geq 2$ for any field $\Bbbk$, moreover,
$\effdim_{\Bbbk}(R_{m,n})= 2$ for any field $\Bbbk$ such that $|\Bbbk|\geq \max(m,n)$.
\item\label{rectband.3} If $m\neq 1$ and $n\neq 1$, then
$\effdim_{\Bbbk}(R_{m,n})\geq 3$ for any field $\Bbbk$, moreover,
$\effdim_{\Bbbk}(R_{m,n})= 3$ for any field $\Bbbk$
such that $|\Bbbk|\geq \max(m,n)$.
\end{enumerate}
\end{proposition}

\begin{proof}
Claim \eqref{rectband.1} is obvious. That $\effdim_{\Bbbk}(R_{m,1})\geq 2$ if $m>1$ and
$\effdim_{\Bbbk}(R_{1,n})\geq 2$ if $n>1$ are also obvious. If $|\Bbbk|\geq \max(m,n)$, then
an effective representation of $R_{m,1}$ by $2\times 2$-matrices over $\Bbbk$ can be obtained by
sending elements of $R_{m,1}$ to different matrices of the form
\begin{equation*}
\left(\begin{array}{cc}1&0\\a&0\end{array}\right), \quad a\in \Bbbk.
\end{equation*}
For $R_{1,n}$ one just transposes the above matrices. Claim \eqref{rectband.2} follows.

Assume now that $m,n\neq 1$. That $\effdim_{\Bbbk}(R_{m,n})\geq 2$ is obvious. We claim that, in fact, $\effdim_{\Bbbk}(R_{m,n})\geq 3$.  It is enough to prove this in the case $m=n=2$.
Assume that we have an effective representation $\varphi$ of $R_{2,2}$ by $2\times 2$ matrices. Then all
these matrices must have the same rank, hence they all have rank one.
 For $i,j=1,2$, let
$V_{(i,j)}$ denote the kernel of the matrix representing $(i,j)$.
It is easy to see that $V_{(1,i)}=V_{(2,i)}=:V_i$ for $i=1,2$ and that $V_1\ne V_2$.  Hence $\Bbbk^{2}=V_{1}\oplus V_{2}$.  This yields a contradiction because one must have
\begin{equation*}
\varphi(1,1)=\left(\begin{array}{cc} 0 & a \\ 0 & 1 \end{array}\right)\qquad \varphi(2,2)=\left(\begin{array}{cc} 1& 0 \\ b & 0 \end{array}\right)
\end{equation*}
and so
\begin{equation*}
\varphi(2,1)=\varphi(2,2)\varphi(1,1)=\left(\begin{array}{cc} 0 & a \\ 0 & ab\end{array}\right).
\end{equation*}
Idempotence implies $ab=1$ and so $\varphi(2,1)=\varphi(1,1)$, a contradiction.
This shows that in the case $m,n>1$ the semigroup $R_{m,n}$ does not have any effective representation
by $2\times 2$ matrices over any field.

Assume that $|\Bbbk|\geq \max(m,n)$, let $a_i$, $i=1,2,\dots,m$, be distinct elements in $\Bbbk$;
and let $b_j$, $j=1,2,\dots,n$, be distinct elements in $\Bbbk$. Then the representation
\begin{displaymath}
(i,j)\mapsto \left(\begin{array}{ccc}1&b_j&0\\0&0&0\\a_i&a_ib_j&0\end{array}\right)
\end{displaymath}
is an effective representation of $R_{m,n}$ by $3\times 3$-matrices over $\Bbbk$.
This completes the proof.
\end{proof}

\subsection{Left regular bands}
Recently, the class of \emph{left regular bands}, that is, bands satisfying the identity $xyx=xy$, has become of importance in algebraic combinatorics and probability, see~\cite{AM,BHR,Br1,Br2,Bj,Sal1,Sal2}.

The set $\Lambda(M)$ of principal left ideals of a left regular band monoid $M$ form a lattice under inclusion, in fact, $Ma\cap Mb=Mab$ for all $a,b\in M$.  Thus $\sigma\colon M\to \Lambda(M)$ given by $\sigma(a)=Ma$ is a homomorphism called the \emph{support map}.  It is the maximal semilattice image homomorphism and also the quotient map of $M$ by its largest $\mathbf {LI}$-congruence.  We recall that a congruence is an \emph{$\mathbf{LI}$-congruence} if whenever $S$ is a congruence class which is a subsemigroup and $e\in S$ is an idempotent, one has $eSe=\{e\}$.  In a left regular band one has     $\mathcal{L}=\mathcal{J}$.

For any field $\Bbbk$,  the induced mapping $\overline{\sigma}\colon \Bbbk M\to \Bbbk \Lambda(M)$ is the semisimple quotient, cf.~\cite{Br1,Br2}.  In particular, the semisimple modules for $M$ are exactly the modules for $\Lambda(M)$.  Thus the simple modules for $\Bbbk$ are one-dimensional and are in bijection with the $\mathcal{L}$-classes of $M$.  The character associated to an $\mathcal{L}$-class $\mathcal{L}_{e}$ sends $m$ to $1$ if $e\in Mm$ and $0$, otherwise.

The following lemma will be used to obtain a lower bound on the effective dimension of a left regular band.

\begin{lemma}\label{lrblower}
Let $M$ be a finite left regular band monoid and let $\Bbbk$ be a field.  If $V$ is an effective module for $M$, then the direct sum of the composition factors of $V$ is an effective module for $\Lambda(M)$.  Moreover, if $M$ does not have a zero element, then the trivial module is also a composition factor of $V$.
\end{lemma}
\begin{proof}
We prove the second statement first.  Let $I$ be the minimal ideal of $M$.  Then all non-trivial simple modules are annihilated by $I$.  It follows that the images of the elements of $I$ are nilpotent under any representation without the trivial representation  as a composition factor.  But $I$ consists of idempotents so $I$ maps to $0$ under any representation without the trivial module as a composition factor.

Assume now that $V$ is effective.  Let $\varphi$ be the representation of $M$ associated to $V$ and let $\rho$ be the representation associated to the direct sum of its  composition factors.  Then $\rho(M)$ is a semilattice.  Thus $\rho$ factors through $\sigma$.  On the other hand, \cite[Lemma 3.1]{AMSV} implies that the congruence associated to $\rho$ is an $\mathbf{LI}$-congruence.  Thus the congruences associated to $\rho$ and $\sigma$ coincide. This establishes the lemma.
\end{proof}

Let $\{+,-,0\}$ be the semigroup where $0$ is the identity and $\{+,-\}$ is a left zero semigroup.   To obtain upper bounds on the effective dimension of left regular bands we need an effective representation of $\{+,-,0\}^{n}$.

\begin{proposition}\label{signseq}
Let $\Bbbk$ be a field.  Then $\{+,-,0\}^{n}$ has an effective representation over $\Bbbk$ of dimension $n+1$.
\end{proposition}
\begin{proof}
Define an action of $M:=\{+,-,0\}^{n}$ on the right of the set $\{0,\ldots,n\}$ by partial functions by putting, for $\alpha=(a_{1},\ldots,a_{n})$ and $i\in \{1,\ldots,n\}$,
\begin{equation*}
i\alpha = \begin{cases} i, & \text{if}\ a_{i}=0;\\ 0,
 & \text{if}\ a_{i}=+\\
\text{undefined,}& \text{if}\ a_{i}=-\end{cases}
\end{equation*}
and putting $0\alpha=0$.  This action is effective and so linearizes to an effective $(n+1)$-dimensional representation of $M$.
\end{proof}

We shall prove shortly that the upper bound in Proposition~\ref{signseq} is tight.  Let us now consider the case of hyperplane face semigroups.  Suppose that $\mathcal{H}$ is a central hyperplane arrangement in $\mathbb R^{n}$.  Each hyperplane $H\in \mathcal{H}$ cuts $\mathbb{R}^{n}$ into two half-spaces $H^{+}$ and $H^{-}$.   Letting $H^{0}=H$, a \emph{face} is a non-empty intersection
\begin{equation*}
F:= \bigcap_{H\in \mathcal{H}}H^{\varepsilon_{H}}
\end{equation*}
where $\varepsilon_{H}\in \{+,-,0\}$.  The \emph{sign sequence} of $F$ is the $\mathcal{H}$-tuple $(\varepsilon_{H})$.  In this way, the set $F(\mathcal{H})$ of all faces of $\mathcal{H}$ can be identified with a subset of $\{+,-,0\}^{\mathcal{H}}$, which in fact is a submonoid called the \emph{face semigroup} of $\mathcal{H}$.  This makes $F(\mathcal{H})$ a left regular band.   See~\cite{AM,Br1,Br2,Sal2} for details, as well as for a geometric description of the multiplication.  The lattice $\Lambda(F(\mathcal{H}))$ can be identified with the \emph{intersection lattice} $L(\mathcal{H})$.  This is the lattice of subspaces of $\mathbb{R}^{n}$ consisting of finite intersections of elements of $\mathcal{H}$ ordered by reverse inclusion.  The support map takes a face to its span.

For example, consider the \emph{Boolean arrangement} in $\mathbb{R}^{n}$ whose hyperplanes are the coordinate hyperplanes (i.e., the orthogonal complements of the coordinate axes). In this case all sign sequences yield a face of the arrangement and so the corresponding hyperplane face semigroup is $\{+,-,0\}^{n}$.  Thus the next result implies that $n+1$ is the effective dimension of $\{+,-,0\}^{n}$.

\begin{theorem}\label{hyperplane}
Let $\mathcal H$ be a central hyperplane arrangement in $\mathbb{R}^{n}$.  Then for any field $\Bbbk$ one has $\effdim_{\Bbbk}(F(\mathcal{H}))=|\mathcal{H}|+1$.
\end{theorem}
\begin{proof}
The lower bound follows from Lemma~\ref{lrblower}.  Indeed, the hyperplanes are the join irreducible elements of the dual of $L(\mathcal{H})$ and the minimal left ideal consists of the chambers and so $F(\mathcal{H})$ has no zero.

On the other hand, Proposition~\ref{signseq} shows that $\{+,-,0\}^{\mathcal{H}}$, and hence $F(\mathcal{H})$, has an effective representation of dimension $|\mathcal{H}|+1$.  This completes the proof.
\end{proof}

For our next result we shall need to recall the description of the projective indecomposable modules for a left regular band monoid $M$ from~\cite{Sal1}.  Let $\mathcal{L}_{e}$ be an $\mathcal{L}$-class of $M$. Then the projective cover of the simple module associated to $\mathcal{L}_{e}$ is the left Sch\"utzenberger representation.  That is, it has basis $\mathcal{L}_{e}$.  If $e\notin Mm$, then $m$ is sent to zero under this representation.  If $e\in Mm$, then $m$ acts on $\mathcal{L}_{e}$ by left multiplication.

We now compute the effective dimension of the free left regular band (monoid) $F_{n}$ on $n$-generators.  Let
 $A$ be a free generating set.  Then $F_{n}$ consists of all injective words $w$ over $A$ including the empty word.  By an \emph{injective word}, we mean one with no repeated letters.  The product is given by concatenation followed by removal of repetitions reading from left to right.  The lattice $\Lambda(F_{n})$ can be identified with the lattice of subsets of $A$ ordered by reverse inclusion.  The map $\sigma$ takes a word to its support, i.e., its set of letters.

\begin{theorem}
Let $\Bbbk$ be a field.  Then
\begin{equation*}
\effdim_{\Bbbk}(F_{n})=\binom{n}{2}+n+1.
\end{equation*}
\end{theorem}
\begin{proof}
To obtain the lower bound, we first observe by Lemma~\ref{lrblower} that any effective module $V$ contains the trivial module as a composition factor.  Let  $A=\{1,\ldots,n\}$ be the free generating set.  The join irreducible elements of the dual of the power set of $A$ ordered by reverse inclusion are the $(n-1)$-element subsets.  So the simple modules corresponding to these $\mathcal{L}$-classes must appear as composition factors in $V$ by Lemma~\ref{lrblower} and by Corollary~\ref{latticecase} (and its proof).

Next we claim that each simple module associated to an $\mathcal{L}$-class corresponding to an $(n-2)$-element subset must be a composition factor of $V$.  There are $\binom{n}{2}$ such simple modules.  Let $u\in F_{n}$ have length $n-2$ and let $a,b$ be the two letters not belonging to $u$.  Then it is straightforward to verify that $x\mapsto xab-xba=xuab-xuba$, for $x\in \mathcal{L}_{u}$, provides an isomorphism of the projective indecomposable module $\Bbbk \mathcal{L}_{e}$ with $\Bbbk F_{n}(uab-uba)$.  Since $V$ is effective, there exists $\alpha\in V$ such that $(uab-uba)\alpha\ne 0$.  Then $x\mapsto x\alpha$ provides a non-zero homomorphism from $\Bbbk F_{n}(uab-uba)$ to $V$.  Since $\Bbbk F_{n}(uab-uba)$ is a projective cover of the simple module associated to $\mathcal{L}_{u}$, this simple module is a composition factor of $V$.  This yields the lower bound.

For the upper bound, it suffices by Proposition~\ref{signseq} to show that $F_{n}$ embeds in $\{+,-,0\}^{\binom{n}{2}+n+1}$.  It is classical that $F_{n}$ embeds in a product of copies of $n$, see~\cite[Proposition 7.3.2]{RS}; we are simply doing the bookkeeping.   Let $B$ be the set of $2$-element subsets of $A$.  Then to each $a\in A$ we associate a function $f_{a}\colon A\cup B\to \{+,-,0\}$ by defining, for $j<k$ from $A$,
\begin{align*}
f_{a}(j)& = \begin{cases}+, & \text{if}\ a= j; \\ 0, & \text{if}\ a\ne j\end{cases}\\
f_{a}(\{j,k\}) &=  \begin{cases}0, & \text{if}\ a\notin \{j,k\}; \\ +, & \text{if}\ a= j;\\ -, &\text{if}\ a=k. \end{cases}
\end{align*}
The map $a\mapsto f_{a}$ extends to a homomorphism $\varphi\colon F_{n}\to \{+,-,0\}^{A\cup B}$ which we claim is injective.  Clearly, $\varphi$ separates words with different support.  If $u,v$ have the same support, then there exist $i<j$ such that $i,j$ are in the support of $u,v$ and appear in a different order in $u,v$.  Thus their images are different in the $\{i,j\}$-coordinate.  This completes the proof.
\end{proof}

We do not know the effective dimension of the free band.

\begin{question}
What is the effective dimension of the free band?
\end{question}

\subsection{The symmetric group $\mathcal{S}_n$}\label{s7.2}

As we have seen above, the classical transformation semigroups $\mathcal{T}_n$, $\mathcal{PT}_n$
and $\mathcal{IS}_n$ all have effective dimension $n$. All these semigroups are natural generalizations of
the symmetric group $\mathcal{S}_n$ and for completeness we recall here the effective dimension of the latter.  The result in characteristic zero is due to Burnside (see~\cite[Chapter~19, \S 8, Theorem~22]{BZ} for a proof) and in positive characteristic to Dickson~\cite{D}.

\begin{theorem}\label{thmsymmgroup}
Assume that the characteristic of $\Bbbk$ does not divide $n$.
Then $\effdim_{\Bbbk}(\mathcal{S}_n)=n-1$.   If $n\ge 5$ and the characteristic of $\Bbbk$ divides $n$, then $\effdim_{\Bbbk}(\mathcal{S}_{n})=n-2$.
\end{theorem}

\subsection{Cyclic semigroups}\label{s7.3}

For $m,n\in\mathbb{N}$, $m\leq n$, denote by $C_{m,n}$
the cyclic semigroup with presentation $\langle
x:x^{n+1}=x^m\rangle$.

\begin{proposition}\label{prop7301}
Let $\Bbbk$ be a field containing a primitive $(n-m+1)^{st}$ root of unity.  Then:
\begin{displaymath}
\effdim_{\Bbbk}(C_{m,n}) =\min(m+1,n).
\end{displaymath}
\end{proposition}

\begin{proof}
The estimate
$\effdim_{\Bbbk}(C_{m,m})\geq m$
follows from Corollary~\ref{cornilp}, while
$\effdim_{\Bbbk}(C_{m,m})\leq m$
follows from Theorem~\ref{geneffdim}.

If $n>m$, then $C_{m,n}$ contains a unique idempotent
$e$, which is of course central. Moreover, $C_{m,n}e\cong \mathbb{Z}/(n-m+1)\mathbb{Z}$ and $C_{m,n}/C_{m,n}e\cong C_{m,m}$.  Let $V$ be an effective module for $C_{m,n}$, which we can view as a module for $C_{m,n}^{1}$.  Then $V=eV\oplus (1-e)V$ and, moreover, $eV$ and $(1-e)V$ are effective modules for $C_{m,n}e$ and $C_{m,n}/C_{m,n}e$, respectively.  Hence $\dim (1-e)V\ge m$ by the previous
paragraph and $\dim eV\ge 1$.
Thus $\effdim_{\Bbbk}(C_{m,n})\geq m+1$.

On the other hand, the representation by
$(m+1)\times (m+1)$ matrices assigning
$x$ the direct sum of the nilpotent Jordan block of
size $m$ with the matrix $(\xi)$, where $\xi$ is a primitive $(n-m+1)^{st}$ root of unity, is effective.
The claim of the proposition follows.
\end{proof}

\section{The table of effective dimensions over $\mathbb{C}$}\label{s9}

We conclude the paper with a table providing the
effective dimension over $\mathbb{C}$ of several
classical families of semigroups.

\begin{center}
\begin{tabular}{|l||c|}
\hline
Semigroup $S$ & $\effdim_{\mathbb{C}}(S)$\\
\hline\hline
Symmetric group $\mathcal{S}_n$ & $n-1$\\
\hline
Full transformation semigroup $\mathcal{T}_n$ & $n$\\
\hline
Full semigroup $\mathcal{PT}_n$ of partial transformations  & $n$\\
\hline
Symmetric inverse semigroup $\mathcal{IS}_n$ & $n$\\
\hline
Full matrix semigroup $\mathrm{Mat}_{n\times n}(\mathbb{F}_q)$ & $(q^n-1)/(q-1)$\\
\hline
Free nilpotent semigroup $N_{m,n}$ of index $n$ & $n$\\
\hline
Nontrivial left- or right zero semigroup & $2$\\
\hline
Rectangular band $R_{m,n}$, $m,n>1$ & $3$\\
\hline
Semigroup $B_n$ of binary relations & $2^n-1$\\
\hline
Path semigroup of an acyclic quiver on $n$ vertices & $n$\\
\hline
Nilpotent cyclic semigroup $C_{m,m}$ & $m$\\
\hline
Cyclic semigroup $C_{m,n}$, $n>m$ & $m+1$\\
\hline
Free left regular band $F_n$ on $n$ generators & $\binom{n}{2}+n+1$\\
\hline
\end{tabular}
\end{center}

\vspace{0.2cm}

\noindent
V.M.: Department of Mathematics, Uppsala University, Box 480,
SE-75106, Uppsala, SWEDEN, e-mail: {\tt mazor\symbol{64}math.uu.se},\\
web: ``http://www.math.uu.se/$\sim$mazor/''
\vspace{0.2cm}

\noindent
B.S.: Department of Mathematics; City College of New York,
Convent Avenue at 138th Street, New York, New York 10031, United States,\\
e-mail: {\tt bsteinberg\symbol{64}ccny.cuny.edu}\\
web: ``http://www.sci.ccny.cuny.edu/$\sim$benjamin/''

\end{document}